\documentclass[12pt]{article}
\usepackage{amsfonts}
\usepackage{amsmath}
\usepackage{mathrsfs}
\usepackage{color}
\usepackage{tikz}
\usepackage{mathrsfs,amscd,amssymb,amsthm,amsmath,bm,graphicx,psfrag,subfigure,url}

\usepackage{indentfirst}

\def \[{\begin{equation}}
\def \]{\end{equation}}

\newtheorem{thm}{Theorem}[section]

\newtheorem{lem}[thm]{Lemma}

\newtheorem{pb}[thm]{Problem}

\newenvironment{wst}
{\setlength{\leftmargini}{1.5\parindent}
 \begin{itemize}
 \setlength{\itemsep}{-1.1mm}}
{\end{itemize}}

\begin{document}
\title{\bf The maximum number of stars in a graph without linear forest}

\author{Sumin Huang\footnote{email: sumin2019@sina.com (S.M.~Huang)},\ \  Jianguo Qian\footnote{Corresponding author, email: jgqian@xmu.edu.cn (J.G,~Qian)}\\
\ \\
School of Mathematical Sciences, Xiamen
University, \\
Xiamen 361005, P.R. China}
\date{}
\maketitle

\begin{abstract}
For two graphs $J$ and $H$, the generalized Tur\'{a}n number, denoted by $ex(n,J,H)$, is the maximum number of copies of $J$ in an $H$-free graph of order $n$. A linear forest $F$ is the disjoint union of paths. In this paper, we determine the number $ex(n,S_r,F)$ when $n$ is large enough and characterize the extremal graphs attaining  $ex(n,S_r,F)$, which generalizes the results on $ex(n, S_r, P_k)$, $ex(n,K_2,(k+1) P_2)$ and $ex(n,K^*_{1,r},(k+1) P_2)$. Finally,  we pose the problem whether the extremal graph for $ex(n,J,F)$ is isomorphic to that for  $ex(n,S_r,F)$, where $J$ is any graph such that the number of $J$'s in any graph $G$ does not decrease by shifting operation on $G$.
\end{abstract}

\vspace{2mm} \noindent{\bf Keywords}: Tur\'{a}n number; Generalized Tur\'{a}n number; Star; Linear forest
\vspace{2mm}

\setcounter{section}{0}
\section{Introduction}\setcounter{equation}{0}
For a graph $G$, the vertex set and edge set of $G$ are denoted as usual by $V(G)$ and  $E(G)$, respectively. For a vertex $v$ of $G$, we denote by $N_G(v)$ the set of the neighbors of $v$ in $G$, that is, the set of the vertices adjacent to $v$. The degree of a vertex $v$ is defined as $d_G(v)=|N_G(v)|$. The minimum degree of the vertices in $G$ is denoted by $\delta(G)$. For a positive integer $n$, we denote by $P_n$, $S_{n-1}$ and $K_n$ the path, star and the complete graph of order $n$, respectively. For two positive integers $s$ and $t$, we denote by  $K_{s,t}$ and $K^*_{s,t}$ the complete bipartite graph and the graph obtained from $K_{s,t}$  by replacing the part of size $s$ by a clique of the same size, respectively.

For two disjoint graphs $G$ and $H$, the \textit{union} of $G$ and $H$, denoted by $G\cup H$, is the graph having vertex set $V(G)\cup V(H)$ and edge set $E(G)\cup E(H)$. Similarly, the \textit{join} of $G$ and $H$ is the graph having vertex set $V(G)\cup V(H)$ and edge set $E(G)\cup E(H)\cup \{(u,v) : u\in V(G), v\in V(H)\}$ and is denoted by $G\vee H$. Further, for any positive integer $k$, we write the disjoint union of $k$ copies of a graph $H$ as $kH$. A graph is called a {\it linear forest} if it is  the  disjoint union of paths or isolated vertices. For a subgraph $H$ of a graph $G$, we denote by $G-H$ the graph obtained from $G$ by removing all vertices in $V(H)$ and by $E(H,G-H)$ the set of the edges joining the vertices between $H$ and $G-H$. For a class  $\mathcal{H}$ of graphs, we say that a graph $G$ is $\mathcal{H}$\textit{-free} if $G$ does not contain any $H\in \mathcal{H}$ as a subgraph (not necessarily induced). In particular, if $\mathcal{H}$ consists of a single graph $H$ then we write $\mathcal{H}$ as $H$ for simplicity.

For a graph $H$, let $ex(n,H)$ denote the maximum number of edges in an $H$-free graph of order $n$. The problem of determining the number $ex(n,H)$ is an elementary topic in extremal graph theory, which originates from the well-known Tur\'{a}n's theorem, i.e., $ex(n,K_{t})=|E(T_{t-1,n})|$, where $T_{t-1,n}$ is the Tur\'{a}n graph. Due to  Tur\'{a}n's work, the number $ex(n,K_{t})$, or more generally the number $ex(n,H)$, is named after him, i.e., the {\it Tur\'{a}n number}. Instead of the number of edges, a natural extension for Tur\'{a}n number is to consider the number of copies of an arbitrary graph. For two graphs $J$ and $H$, let $ex(n,J,H)$ denote the maximum number of copies of $J$  in an $H$-free graph of order $n$. In earlier 1962, Erd\H{o}s \cite{Erdos} determined the number $ex(n,K_m,K_{t})$ for any positive integers $m$ and $t$. In 2008, Bollob\'{a}s and Gy\H{o}ri \cite{Bollobas} estimated that $ex(n, K_3, C_5)=\Theta(n^{3/2})$. In general,  Alon and Shikhelman \cite{Alon} estimated that $ex(n, J,H)=\Theta(n^m)$, where $J$ and $H$ are any two graphs and  $m$ is an integer depending on $J$ and $H$.

In the past ten years or more, in addition to complete graphs, the linear forests received much attention. Indeed, this topic may date back to an earlier result given by Erd\H{o}s and Gallai \cite{Erdos2} which, in terms of Tur\'{a}n number, determines $ex(n,K_2,M_{k+1})$, where $M_{k+1}$ is the matching with $k+1$ edges, i.e., the linear forest of $k+1$ edges. In \cite{Lidicky}, Lidick\'{y} et al. determined the number  $ex(n,K_2, F)$ when $n$ is large enough, where $F$ is an arbitrary linear forest and $F\neq t P_3$. Later in \cite{Yuan}, Yuan and Zhang determined $ex(n,K_2,t P_3)$ for all $n$ and $t$, which, together with the former result of  Lidick\'{y} et al., gives the complete determination of the number $ex(n,K_2,F)$ for any linear forest $F$. Recently,  Gy\H{o}ri et. \cite{Gyori} determined the values of $ex(n, C_4, P_k),ex(n, P_3, P_k)$ and $ex(n, S_r, P_k)$ for $r\geq 2, k\geq 3$ and sufficiently large $n$.  By using shifting method, Wang \cite{Wang} determined $ex(n,J,M_{k+1})$ where $J$ is one of  $K_s$ and $K^*_{s,t}$, which extends the result on $ex(n,K_2,M_{k+1})$ of Erd\H{o}s and Gallai. More recently, Zhu et al. \cite{Zhu} determined the number $ex(n,K_s,F)$ for any $s\geq 2$ and any forest $F$ consisting of paths of even order.

Recently, instead of a single graph, research interests of Tur\'{a}n problem also focused on a class of graphs. For a class $\mathcal{H}$ of graphs, the {\it generalized Tur\'{a}n number}  $ex(n,J,\mathcal{H})$ is defined as the  maximum number  of copies of $J$ in an $\mathcal{H}$-free graph of order $n$. For two positive integers $n$ and $k$, let $\mathcal{L}_{n,k}$ be the class of all linear forests with $n$ vertices and $k$ edges and let $\mathcal{L}_{n}^{k}=\cup_{i=n-k+1}^{n-1} \mathcal{L}_{n,i}$. The  number of $ex(n,K_2,\mathcal{L}_{n}^{k})$ is determined by Wang and Yang in \cite{Wang2} for $n\geq 3k$. Further,  Ning and Wang \cite{Ning} pointed out that $ex(n,K_2,\mathcal{L}_{n}^{k})=ex(n,K_2,\mathcal{L}_{n,n-k+1})$ and determined the value of $ex(n,K_2, \mathcal{L}_{n,k})$ for any $n$ and $k$. More recently, using shifting method, Zhang et al. \cite{Zhang} generalized the above two results by extending $K_2$ to $K_s$ and $K^*_{s,t}$.


\begin{figure}[htp]
\begin{center}
\includegraphics[width=120mm]{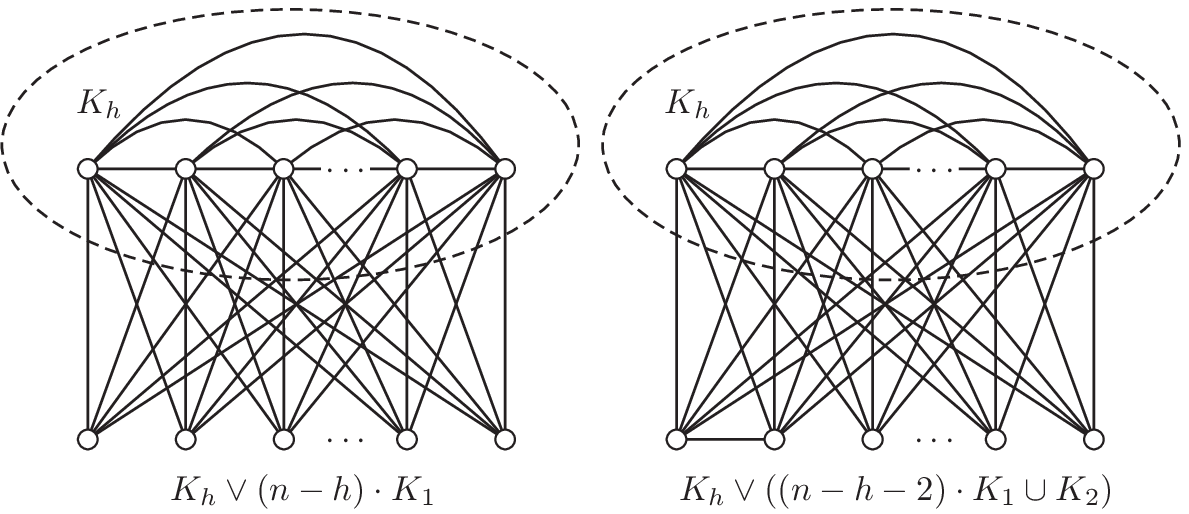}
\caption{$G_F(n)$}\label{fig1}
\end{center}
\end{figure}

Given $t$ positive integers $k_1, k_2,\ldots ,k_t$, let $h=-1+\sum_{i=1}^{t}\left\lfloor \frac{k_i}{2}\right\rfloor$ and $F=\cup_{i=1}^{t} P_{k_i}$. If at least one of $k_1, k_2,\ldots ,k_t$ is not equal to $3$, then let $G_F(n)=K_h\vee (n-h) K_1$ if at least one $k_i$ is even and $G_F(n)=K_h\vee ((n-h-2) K_1\cup K_2)$ otherwise, see Figure \ref{fig1}. In particular, let $G_{tP_3}(n)=K_h\vee \frac{n-h}{2}K_2$ if $(n-h)$ is even and $G_{tP_3}(n)=K_h\vee (\frac{n-h-1}{2}K_2\cup K_1)$ otherwise. Since each path $P_k$ in $G_F(n)$ covers at least $\left\lfloor k/2 \right\rfloor$ vertices in $K_h$, $G_F(n)$ is an $F$-free graph. Similarly, $G_{tP_3}(n)$ is a $tP_3$-free graph.

In this paper, we determine the number $ex(n, S_r, \cup_{i=1}^{t} P_{k_i})$ for $n$ large enough and show that $G_F(n)$ is the unique extremal graph if $F\neq tP_3$ and $G_{tP_3}(n)$ is the unique extremal graph if $F=tP_3$, which generalizes the result on $ex(n, S_r, P_k)$ given in \cite{Gyori}. Moreover, since $M_{k+1}=(k+1) P_2$ is a linear forest, our result implies that of Erd\H{o}s and Gallai \cite{Erdos2} on $ex(n,K_2,M_{k+1})$ and Wang \cite{Wang} on $ex(n,K^*_{1,r},M_{k+1})$.

\begin{thm}\label{th1}
Let $k_1, k_2,\ldots ,k_t$ and $r$ be integers no less than $2$, $h=\sum_{i=1}^{t}\left\lfloor \frac{k_i}{2}\right\rfloor-1$ and $F=\cup_{i=1}^{t} P_{k_i}$. Suppose $n$ is sufficiently large.
\begin{wst}
\item[{\rm 1).}] If at least one of $k_1, k_2,\ldots ,k_t$ is not equal to $3$, then
$$ex(n,S_r,\bigcup_{i=1}^{t} P_{k_i})=h \binom{n-1}{r}+(n-h) \binom{h}{r}+2 \eta_F \binom{h}{r-1},$$
where $\eta_F=1$ if all $k_i$ are odd and $\eta_F=0$ otherwise. Further, $G_F(n)$ is the unique extremal graph attains $ex(n,S_r,\cup_{i=1}^{t} P_{k_i})$.

\item[{\rm 2).}]
$$ex(n,S_r,tP_3)=h \binom{n-1}{r}+(n-h) \binom{h+1}{r}+\tau_{n,h} \left (\binom{h}{r}-\binom{h+1}{r}\right ),$$
where $\tau_{n,h}=0$ if $2|(n-h)$ and $\tau_{n,h}=1$ otherwise. Further, $G_{tP_3}(n)$ is the unique extremal graph attains $ex(n,S_r,tP_3)$.
\end{wst}
\end{thm}

At the end of the article, we pose the problem whether the extremal graph for $ex(n,J,F)$ is isomorphic to that for  $ex(n,S_r,F)$, where $J$ is any graph such that the number of $J$'s in any graph $G$ does not decrease by shifting operation on $G$.

\section{Lemmas}
In this section, we introduce some necessary lemmas for the proof of Theorem \ref{th1}.

Following  \cite{Chen}, we define  some special graphs. For two positive integers $t_1$ and $t_2$, let $L_{t_1,t_2,h,h+1}=K_1\vee (t_1 K_h \cup t_2 K_{h+1})$ and call the only vertex of $K_1$ the center. In particular, we write $L_{t,0,h,h+1}=L_{t,h}$. Let $F_{t_1,t_2,h,h+1}$ be the graph of order $t_1h+(t_2+1)(h+1)+1$ obtained from $L_{t_1,t_2,h,h+1}$ and $K_{h+1}$ by adding an edge joining the center of $L_{t_1,t_2,h,h+1}$ and a vertex of $K_{h+1}$. Further, for $h\geq 2$, let $T_{t_1,t_2,h,h+1}$ be the graph of order $t_1h+(t_2+2)(h+1)+1$ obtained from $2K_{h+1}$ and $L_{t_1,t_2,h,h+1}$ by adding an edge joining the center of $L_{t_1,t_2,h,h+1}$ and a vertex of each $K_{h+1}$, respectively. For $n\geq 7$, let $H_n^1$ be the graph of order $n$ obtained from  $K_2\vee (n-4) K_1$ and $K_3$ by identifying a vertex in $K_2$ of $K_2\vee (n-4) K_1$ with a vertex in $K_3$. For $n\geq 9$, let $H_n^2$ be the graph of order $n$ obtained from  $H_{n-2}^1$ and $K_3$ by identifying a vertex that has the second largest degree in $H_{n-2}^1$, with a vertex in $K_3$. Let $U_{3,h}$ be the graph of order $3h+3$ obtained from $3K_{h+1}$ by adding three new edges $uv,vw$ and $uw$, where $u,v,w$ are in the three $K_{h+1}$'s, respectively.

\begin{lem}\label{lem0}\cite{Chen}
Let $G$ be a graph of order $n$, $F=P_a\cup P_b$ and $h=\left\lfloor \frac{a}{2}\right\rfloor+\left\lfloor \frac{b}{2}\right\rfloor-1$, where $a$ and $b$ are two positive integers.  If $\delta(G)\geq h$, then $F\subseteq G$, unless one of the following holds.
\begin{wst}
\item[{\rm 1).}] $G\subseteq G_F(n)$;
\item[{\rm 2).}] $a$ and $b$ are both even and $a=b,G\cong L_{t,h},n=th+1$;
\item[{\rm 3).}] $|a-b|=1,G\cong L_{t,h}$ and $n=th+1$;
\item[{\rm 4).}] $a$ and $b$ are both odd and $|a-b|=2,G\cong L_{t,h},n=th+1$;
\item[{\rm 5).}] $a$ and $b$ are both odd and $a=b,G\cong U_{3,h},n=3h+3$;
\item[{\rm 6).}] $a$ and $b$ are both odd and $a=b,G\subseteq L_{t_1,t_2,h,h+1},n=t_1 h+t_2 (h+1)+1$;
\item[{\rm 7).}] $a$ and $b$ are both odd and $a=b,G\subseteq F_{t_1,t_2,h,h+1},n=t_1 h+(t_2+1)(h+1)+1$;
\item[{\rm 8).}] $a$ and $b$ are both odd and $a=b,G\subseteq T_{t_1,t_2,h,h+1},n=t_1 h+(t_2+2)(h+1)+1$;
\item[{\rm 9).}] $n$ is even, $G\subseteq K_2 \vee \left(\frac{n-2}{2}\right) K_2$ and either $\{a,b\}=\{6,3\}$ or $\{7,3\}$;
\item[{\rm 10).}] $n$ is odd, $\{a,b\}=\{9,3\}$ and $G\subseteq K_3 \vee \left(\frac{n-3}{2}\right) K_2$;
\item[{\rm 11).}] $\{a,b\}=\{5,3\}$ and either $G\subseteq H_n^1$ or $G\subseteq H_n^2$.
\end{wst}
\end{lem}

\begin{lem}\label{lemn1}\cite{Gyori}
Let $k$ be an integer with $k\geq 3$ and let $h=\left\lfloor \frac{k}{2}\right\rfloor-1$. If $n$ is sufficiently large, then
$$ex(n,S_r,P_k)=h \binom{n-1}{r}+(n-h) \binom{h}{r}+2 \eta_k \binom{h}{r-1},$$
where $\eta_k=1$ if $k$ is odd and $\eta_k=0$ otherwise.
\end{lem}

For two graphs $J$ and $G$, we denote by  $\mathcal{N}(J,G)$ the number of copies of $J$ in $G$. In particular, $\mathcal{N}(S_0,G)=|V(G)|$. We note that the number of $S_r$'s in any graph $G$ equals $\sum_{v\in V(G)}\binom{d(v)}{r}$. So by the definitions of $G_F(n)$ and $G_{tP_3}(n)$, the following two lemmas are immediate.
\begin{lem}\label{lemn}
Let  $k_1, k_2,\ldots ,k_t$ be $t$ positive integers, $F=\cup_{i=1}^{t} P_{k_i}$ and $h=\sum_{i=1}^{t} \left\lfloor \frac{k_i}{2}\right\rfloor-1$. If $n$ is sufficiently large, then
$$\mathcal{N}(S_r, G_F(n))=h \binom{n-1}{r}+(n-h) \binom{h}{r}+2 \eta_F \binom{h}{r-1},$$
where $\eta_F=1$ if all $k_i$ are odd and $\eta_F=0$ otherwise.
\end{lem}
\begin{lem}\label{lemnp}
Let  $t$ be an positive integers and $h=t-1$. If $n$ is sufficiently large, then
$$\mathcal{N}(S_r, G_{tP_3}(n))=h \binom{n-1}{r}+(n-h) \binom{h+1}{r}+\tau_{n,h} \left (\binom{h}{r}-\binom{h+1}{r}\right ),$$
where $\tau_{n,h}=0$ if $(n-h)$ is even and $\tau_{n,h}=1$ otherwise.
\end{lem}

In the following, we introduce an operation that will be used in our forthcoming discussion, which is also frequently used in extremal graph theory and  extremal set theory. Let $G$ be a graph with vertex set $\{1,2,\cdots,n\}$. For a vertex $i$ and an edge $e$, by $i\in e$ we mean that $i$ is an end vertex of $e$. Further,  for two vertices $i,j$ and an edge $e$ with $j\in e$  and $i\notin e$, by $(e-\{j\})\cup\{i\}$ we mean the edge obtained from $e$ by replacing the end vertex $j$ by $i$. For two integers $i$ and $j$ with $1\leq i<j\leq n$ and $e\in E(G)$, we define the $ij$-\textit{shifting} operation on $e$ as follows:
$$ S_{ij}(e)=\left\{
\begin{aligned}
& (e-\{j\})\cup\{i\}, ~ \text{if}~ j\in e, ~ i\notin e ~ \text{and} ~ (e-\{j\})\cup\{i\}\notin E(G);\\
& e,~ \text{otherwise.} \\
\end{aligned}\right.
$$
Define $S_{ij}(G)$ to be a graph with vertex set $V(G)$ and edge set $\{S_{ij}(e): e\in E(G)\}$. We show that a shifting operation does not reduce the number of $S_r$'s.

\begin{lem}\label{lem3}
For any two vertices $i$ and $j$ with $1\leq i<j\leq n$, $\mathcal{N}(S_r,G)\leq \mathcal{N}(S_r,S_{ij}(G))$, with equality if and only if $S_{ij}(G)\cong G$.
\end{lem}
\begin{proof}
Let $n_i=|N_G(i)\backslash N_G(j)|$, $n_j=|N_G(j)\backslash N_G(i)|$ and $n_{ij}=|N_G(i)\cap N_G(j)|$. By the definition of $ij$-shifting, we notice that the degree of the vertex $i$ changes from $n_i+n_{ij}$ to $n_i+n_{ij}+n_j$,   $j$ changes from $n_i+n_{ij}$ to $n_{ij}$ and all other vertices does not change. Hence,
\begin{align}
\mathcal{N}(S_r,S_{ij}(G))-\mathcal{N}(S_r,G)=\binom{n_i+n_j+n_{ij}}{r}+\binom{n_{ij}}{r}-\binom{n_i+n_{ij}}{r}-\binom{n_j+n_{ij}}{r}.\label{eqs}
\end{align}
If $n_i>0$ and $n_j>0$, then \eqref{eqs} implies $\mathcal{N}(S_r,S_{ij}(G))>\mathcal{N}(S_r,G)$. If $n_i=0$ or $n_j=0$, then by the definition of $S_{ij}(G)$, $S_{ij}(G)\cong G$ and $\mathcal{N}(S_r,G)= \mathcal{N}(S_r,S_{ij}(G))$, as desired.
\end{proof}

\begin{lem}\label{lem2}
Let $a$ and $b$ be two positive integers with $(a,b)\neq(3,3)$, $F=P_a\cup P_b$ and $h=\left\lfloor \frac{a}{2}\right\rfloor+\left\lfloor \frac{a}{2}\right\rfloor-1$. For any integer $r$ with $r\geq 2$, if $G$ is an $F$-free graph of order $n$, then $\mathcal{N}(S_r,G)\leq\mathcal{N}(S_r,G_{F}(n))$ when $n$ is large enough, with equality if and only if $G\cong G_F(n)$.
\end{lem}
\begin{proof}
Let $G$ be a graph satisfying $\mathcal{N}(S_r,G)=ex(n,S_r,F)$.  We proceed by considering the following two possible cases.

{\bf{Case 1.}} $\delta(G)\geq h$.

Since $G$ is $F$-free and $\mathcal{N}(S_r,G)=ex(n,S_r,F)$, in order to show that $G\cong G_F(n)$, it suffices to compare the numbers of $S_r$ in those graphs as indicated in the 11 cases of Lemma \ref{lem0}.

If $h=1$, then $(a,b)=(2,2)$ or $(2,3)$ or $(3,2)$. So by Lemma \ref{lem0}, either $G\subseteq G_F(n)$ or $G\cong L_{t,1}$. In this case, $L_{t,1}\cong G_F(n)$ and, hence, the lemma follows immediately.

Now we consider $h>1$. We notice that  the center of $L_{t,h}$ has degree $n-1$ while all other vertices have degree $h$. This means that $\mathcal{N}(S_r,L_{t,h})=\binom{n-1}{r}+(n-1)\binom{h}{r}$. So by Lemma \ref{lemn}, one has
\begin{align*}
\mathcal{N}(S_r, G_F(n))-\mathcal{N}(S_r,L_{t,h})&=~(h-1) \binom{n-1}{r}+(1-h) \binom{h}{r}+2 \eta_F \binom{t}{r-1}\\
&=~(h-1) n^r + o(n^{r-1}),
\end{align*}
which is greater than $0$ when $n$ is large enough.

Similarly, it can be verified that
\begin{align*}
\mathcal{N}(S_r,U_{3,h})=&~3 \binom{h}{r}+ 6 \binom{h}{r-1}+3 \binom{h}{r-2}+ 3h \binom{h}{r},\\
\mathcal{N}(S_r,L_{t_1,t_2,h,h+1})=&~t_1 h \binom{h}{r}+ t_2 (h+1) \binom{h+1}{r}+ \binom{n-1}{r},\\
\mathcal{N}(S_r,F_{t_1,t_2,h,h+1})=&~(t_1+1) h \binom{h}{r}+ (t_2 (h+1)+1) \binom{h+1}{r}+ \binom{n-h-1}{r},\\
\mathcal{N}(S_r,T_{t_1,t_2,h,h+1})=&~(t_1+2) h \binom{h}{r}+ (t_2 (h+1)+2) \binom{h+1}{r}+ \binom{n-2h-1}{r}.
\end{align*}
For each $H\in \{U_{3,h}, L_{t_1,t_2,h,h+1}, F_{t_1,t_2,h,h+1}, T_{t_1,t_2,h,h+1}\}$, since the coefficient of the highest power of $n$, i.e., $n^r$, in $\mathcal{N}(S_r, G_F(n))$ is greater than that in $\mathcal{N}(S_r, H)$, it follows that $\mathcal{N}(S_r, G_F(n))> \mathcal{N}(S_r, H)$ when $n$ is large enough.

Finally, If $\{a,b\}=\{6,3\}$ or $\{7,3\}$ and $n$ is even, then $h=3$. In this case, $\mathcal{N}(S_r, G_F(n))$ $=3 \binom{n-1}{r}+(n-3) \binom{3}{r}+2 \eta_F \binom{3}{r-1}$. Note that $\mathcal{N}(S_r, K_2 \vee \left(\frac{n-2}{2}\right) K_2)= 2 \binom{n-1}{r}+(n-2) \binom{3}{r}$. It is clear that $\mathcal{N}(S_r, G_F(n))>\mathcal{N}(S_r, K_2 \vee \left(\frac{n-2}{2}\right) K_2)$ when $n$ is large enough. Similarly, we have $\mathcal{N}(S_r, G_F(n))>\mathcal{N}(S_r, K_2 \vee \left(\frac{n-2}{2}\right) K_2)$ (resp. $\mathcal{N}(S_r, G_F(n))>\mathcal{N}(S_r, H_n^1)$ and $\mathcal{N}(S_r, G_F(n))>\mathcal{N}(S_r, H_n^2)$) when $\{a,b\}=\{9,3\}$ (resp. $\{a,b\}=\{5,3\}$) and $n$ is large enough.

So by Lemma  \ref{lem0}, we have $G\subseteq G_F(n)$. Further, notice that $\mathcal{N}(S_r, H)<\mathcal{N}(S_r, G_F(n))$ for any proper subgraph $H$ of $G_F(n)$. This implies that $G\cong G_F(n)$ since $G$ attains $ex(n,S_r,F)$ and $G_F(n)$ is $F$-free.

{\bf{Case 2.}} $\delta(G)<h$.

Let $v$ be a vertex of $G$ with $d(v)=\delta(G)<h$. Note that each vertex has degree at most $n-1$. We have
\begin{align*}
\mathcal{N}(S_r,G)-\mathcal{N}(S_r,G-v)&=\binom{d(v)}{r}+\sum_{u\in N(v)}\binom{d(u)}{r-1}\\
&<~\binom{h}{r}+h\binom{n-1}{r-1}\\
&=~\mathcal{N}(S_r,G_F(n))-\mathcal{N}(S_r,G_F(n-1)),
\end{align*}
where the last equality holds by  Lemma \ref{lemn} and a direct calculation. Since $G-v$ is still an $F$-free graph of order $n-1$, we have $\mathcal{N}(S_r,G-v)\leq ex(n-1,S_r,F)$. Hence,
\begin{align}
ex(n,S_r,F)-ex(n-1,S_r,F)<\mathcal{N}(S_r,G_F(n))-\mathcal{N}(S_r,G_F(n-1)).\label{eq1}
\end{align}

Let $f(n)=ex(n,S_r,F)-\mathcal{N}(S_r,G_F(n))$. It is clear that $f(n)\geq 0$. To prove our result, it is sufficient to show that $f(n)=0$ and $G\cong G_F(n)$ for large enough $n$. \eqref{eq1} implies that $f(n)< f(n-1)$. Since $f(n-1)$ is finite and $f(i)\geq 0$ for any positive integer $i$, there must exist an integer $m$ with $m\geq n$ such that $f(m)=0$.  This means that $G_F(m)$ attains $ex(m,S_r,F)$. Let $G_m$ be an arbitrary extremal graph of order $m$ that attains $ex(m,S_r,F)$. We must have $\delta(G_m)\geq h$, for otherwise, we would have $f(m+1)<f(m)=0$, a contradiction. So by the discussion in Case 1, $G_m\cong G_F(m)$, as desired.
\end{proof}

\begin{lem}\label{lem2p}
Let $F=2P_3$. For any integer $r$ with $r\geq 2$, if $G$ is an $F$-free graph of order $n$, then $\mathcal{N}(S_r,G)\leq\mathcal{N}(S_r,G_{2P_3}(n))$ when $n$ is large enough, with equality if and only if $G\cong G_{2P_3}(n)$.
\end{lem}
\begin{proof}
Let $G$ be an $F$-free graph satisfying $\mathcal{N}(S_r,G)=ex(n,S_r,F)$.

If $G$ is connected, then $\delta(G)\geq 1$.
By Lemma \ref{lemn}, $G\in \{ G_F(n), U_{3,1}, L_{t_1,t_2,1,2}, F_{t_1,t_2,1,2},$ $T_{t_1,t_2,1,2}\}$. Then by similar argument in Lemma \ref{lem2}, we have $G\cong L_{t_1,t_2,1,2}$ for some $t_1$ and $t_2$ with $n=t_1+2t_2+1$. Since $\mathcal{N}(S_r,L_{t_1-2,t_2+1,1,2})\leq \mathcal{N}(S_r,L_{t_1,t_2,1,2})$, $G\cong L_{0,\frac{n-1}{2},1,2}$ if $n$ is odd and $G\cong L_{1,\frac{n-2}{2},1,2}$ if $n$ is even. This implies $G\cong G_{2P_3}(n)$.

If $G$ is disconnected, then let $G_1,\ldots,G_s$  be the components of $G$ with $n_1,\ldots,n_s$ vertices, respectively. In this case, $G_i\cong G_{2P_3}(n_i)$. Let $u_i$ be a vertex with the maximum degree in $G_i$. We can use shifting operation to obtain a graph $G'=G-\sum_{i=2}^{s}\{u_iv:v\in N_G(u_i)\}+\sum_{i=2}^{s}\{u_1v:v\in N_G(u_i)\}$. Note that $G'$ is still an $F$-free graph. By Lemma \ref{lem3}, $\mathcal{N}(S_r,G)<\mathcal{N}(S_r,G')$, which is a contradiction to $\mathcal{N}(S_r,G)=ex(n,S_r,F)$.
\end{proof}

\section{The proof of Theorem \ref{th1}}

In this section, we give the proof of Theorem \ref{th1}.

1). Let $G$ be an $F$-free graph satisfying $\mathcal{N}(S_r,G)=ex(n,S_r,F)$. To prove our result, it is sufficient to show $G\cong G_F(n)$ when $n$ is large enough. We apply induction on $t$. If $t=2$, then 1) follows directly by Lemma \ref{lem2}. We now assume that 1) holds for $t-1$.

Without loss of generality, assume that $k_t$ is a smallest even number in $\{k_1, k_2,\ldots, k_{t}\}$ if every $k_i$ is even or a smallest odd number otherwise. One can see that at least one of $k_1, k_2,\ldots, k_{t-1}$ is not equal to $3$ as $k_i\not=1$ for all $i\in\{1,2,\ldots,t\}$. Further, $k_1, k_2,\ldots, k_{t-1}$ are all odd if $k_1, k_2,\ldots, k_{t}$ are all odd and at least one of $k_1, k_2,\ldots, k_{t-1}$ is even otherwise. Since $t\geq 2$, Lemma \ref{lemn1} and Lemma \ref{lemn} implies that $\mathcal{N}(S_r,G)=ex(n,S_r,F)\geq \mathcal{N}(S_r,G_F(n))> ex(n,S_r,P_{k_t})$, meaning that $G$ contains $P_{k_t}$ as a subgraph. Let $P$ be the subgraph of $G$  induced by $V(P_{k_t})$. Now we give the following claim.
\begin{lem}\label{cla1}
If $n$ is large enough, then $V(P)$ contains a subset $U$ with $|U|=\left\lfloor k_t/2\right\rfloor$ such that
\begin{equation}\label{U}
\left|\left(\bigcap_{u\in U}N_{G}(u)\right)\cap V(G-P)\right|\geq k_1+k_2+\cdots+k_t.
\end{equation}
Moreover, $G-U$ is $(F-P_{k_t})$-free.
\end{lem}
\begin{proof}
For an integer $k$, let $\mathcal{N}_k(S_r,G-P,P)$ denote the number of those $S_r$'s which have the center in $G-P$ and exactly $k$ leaves (i.e., the vertices of degree one) in $P$. We define $\mathcal{N}_k(S_r,P,G-P)$ similarly. By a direct calculation, one has
\begin{align}
&~\mathcal{N}_r(S_r,G-P,P)+\mathcal{N}_r(S_r,P,G-P)\notag\\
=&~\mathcal{N}(S_r,G)-\mathcal{N}(S_r,G-P)-\mathcal{N}(S_r,P)-\sum_{i=1}^{r-1}\mathcal{N}_{i}(S_r,G-P,P)-\sum_{i=1}^{r-1}\mathcal{N}_{i}(S_r,P,G-P).\label{eq2}
\end{align}

Since $G-P$ is an $(F-P_{k_t})$-free graph of order $n-k_t$, one has $\mathcal{N}(S_r,G-P)\leq ex(n-k_t,S_r,F-P_{k_t})=\mathcal{N}(S_r,G_{F-P_{k_t}}(n-k_t))$ by the induction hypothesis. Further, it is clear that $\mathcal{N}(S_r,P)\leq \mathcal{N}(S_r,K_{k_t})= k_t\binom{k_t-1}{r}$.

For $i\in\{1,2,\cdots,r-1\}$, we notice that the number of $S_{r-i}$'s in $G-P$ is at most $ex(n-k_t,S_{r-i},F-P_{k_t})$ and the number of vertices in $P$ is $k_t$. So by the definition of $\mathcal{N}_{i}(S_r,G-P,P)$, we have $\mathcal{N}_{i}(S_r,G-P,P)\leq \binom{k_t}{i}ex(n-k_t,S_{r-i},F-P_{k_t})$. So by the induction hypotheses, $\mathcal{N}_{i}(S_r,G-P,P)\leq \binom{k_t}{i}\mathcal{N}(S_{r-i},G_{F-P_{k_t}}(n-k_t))$. Similarly, $\mathcal{N}_{i}(S_r,P,G-P)\leq k_t \binom{n-k_t}{i} \binom{k_t-1}{r-i}$. Recall that $\mathcal{N}(S_r,G)= ex(n,S_r,F)\geq \mathcal{N}(S_r,G_F(n))$.

Combining the argument above with \eqref{eq2} yields that
\begin{align}
&~\mathcal{N}_r(S_r,G-P,P)+\mathcal{N}_r(S_r,P,G-P)\notag\\
\geq&~\mathcal{N}(S_r,G_F(n))-\mathcal{N}(S_r,G_{F-P_{k_t}}(n-k_t))-k_t \binom{k_t-1}{r}\notag\\
&~-\sum_{i=1}^{r-1}\binom{k_t}{i}\mathcal{N}(S_{r-i},G_{F-P_{k_t}}(n-k_t))-\sum_{i=1}^{r-1} k_t \binom{n-k_t}{i} \binom{k_t-1}{r-i}\notag\\
=&~\frac{1}{r!}\left(\left\lfloor \frac{k_1}{2}\right\rfloor +\left\lfloor \frac{k_2}{2}\right\rfloor+\cdots+\left\lfloor \frac{k_t}{2}\right\rfloor -1\right) n^r\notag\\
&~- \frac{1}{r!}\left(\left\lfloor \frac{k_1}{2}\right\rfloor +\left\lfloor \frac{k_2}{2}\right\rfloor+\cdots+\left\lfloor \frac{k_{t-1}}{2}\right\rfloor -1\right) n^r+ o(n^{r-1})\notag\\
=&~\frac{1}{r!}\left\lfloor \frac{k_t}{2}\right\rfloor n^r+o(n^{r-1}),\label{eq3}
\end{align}
where the last second equality holds by Lemma \ref{lemn}.

We now consider to derive an upper bound on $\mathcal{N}_r(S_r,G-P,P)+\mathcal{N}_r(S_r,P,G-P)$. Let $n_0$ be the number of those vertices $v$ in $G-P$ satisfying $|N_G(v)\cap V(P)|\geq \left\lfloor {k_t}/{2}\right\rfloor$. Let $G'=(X,Y)$ be the bipartite graph obtained from $G$ by deleting all edges in $E(P)$ and $E(G-P)$. Then $\mathcal{N}_r(S_r,G-P,P)+\mathcal{N}_r(S_r,P,G-P)=\mathcal{N}(S_r,G')$. We label $X$ as $\{1,2,\cdots,k_t\}$ and $Y$ as $\{k_t+1,k_t+2,\cdots,n\}$ so that $d_{G'}(i)\leq d_{G'}(j)$ for $k_t+1\leq i<j\leq n$. For the sake of brevity, let $X'=\{\left\lceil k_t/2\right\rceil +2,\cdots,k_t\}$ and $Y'=\{k_t+1,\cdots,n-n_0\}$.

By the definition of $n_0$, each vertex in $Y'$ has degree at most $\left\lfloor k_t/2\right\rfloor -1$. If  $Y'$ has a vertex $u$ satisfying $N_{G'}(u)\nsubseteq X'$, then we choose $j\in N_{G'}(u)\backslash X'$ and $i\in X'\backslash N_{G'}(u)$. By Lemma \ref{lem3}, $\mathcal{N}(S_r,G')\leq \mathcal{N}(S_r,S_{ij}(G'))$. Hence we can obtain a graph $G''$ by shifting operations such that $N_{G''}(u)\subseteq X'$ for all $u\in Y'$ and $\mathcal{N}(S_r,G')\leq \mathcal{N}(S_r,G'')$.

\begin{figure}[htp]
\begin{center}
\includegraphics[width=120mm]{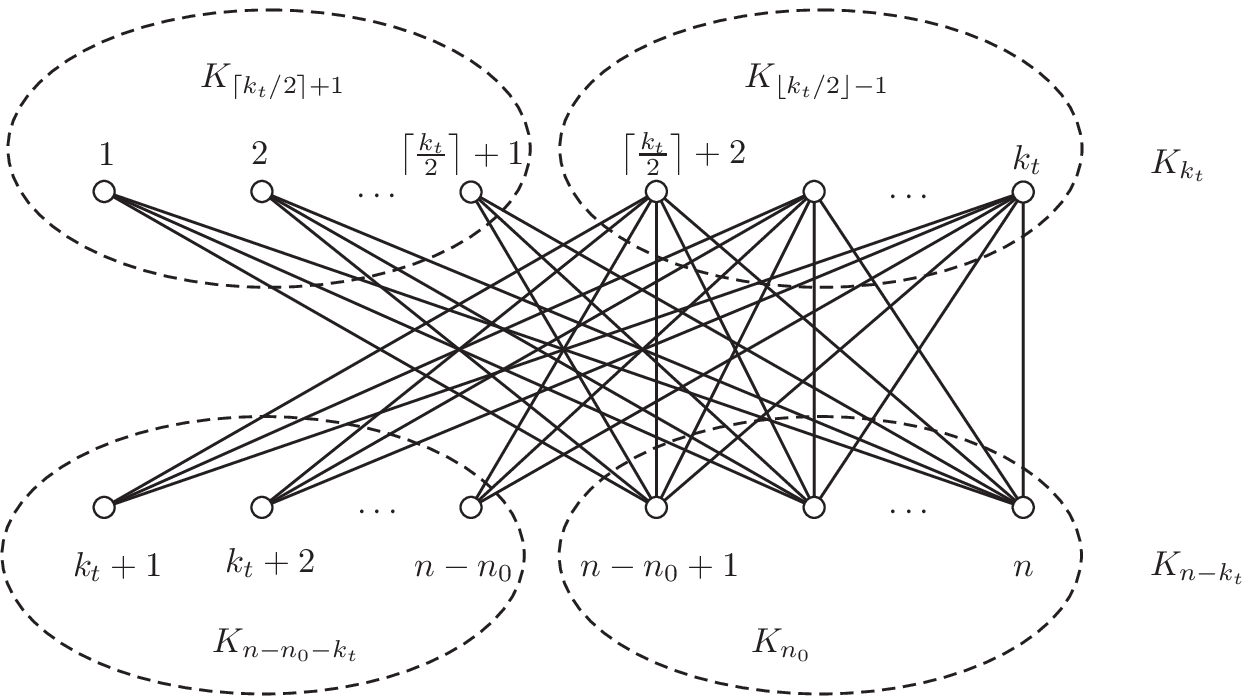}
\caption{$H$}\label{fig2}
\end{center}
\end{figure}

Let $H$ be the graph obtained from the complete bipartite graph $K_{k_t,n-k_t}$ with vertex set $\{1,2,\ldots,n\}$ by deleting the edges in $\{(i,j)\in E(K_{k_t,n-k_t}): 1\leq i\leq \left\lceil k_t/2\right\rceil+1 \quad \text{and} \quad k_t+1\leq j\leq n-n_0\}$, see Figure \ref{fig2}. It clear that $G''\subseteq H$ and, hence, $\mathcal{N}(S_r,G'')\leq \mathcal{N}(S_r,H)$. This means that $\mathcal{N}_r(S_r,G-P,P)+\mathcal{N}_r(S_r,P,G-P)=\mathcal{N}(S_r,G')\leq \mathcal{N}(S_r,G'')\leq \mathcal{N}(S_r,H)$. Therefore,
\begin{align}
&~\mathcal{N}_r(S_r,G-P,P)+\mathcal{N}_r(S_r,P,G-P)\notag\\
\leq &~n_0 \binom{k_t}{r}+ \left(n-k_t-n_0\right) \binom{\left\lfloor \frac{k_t}{2}\right\rfloor-1}{r}+ \left(\left\lfloor \frac{k_t}{2} \right\rfloor-1\right) \binom{n-k_t}{r}+ \left(\left\lceil \frac{k_t}{2}\right\rceil +1\right) \binom{n_0}{r}\notag\\
\leq &~ \left(\binom{k_t}{r}-\binom{\left\lfloor \frac{k_t}{2}\right\rfloor -1}{r}\right) n_0+\frac{1}{r!}\left(\left\lceil \frac{k_t}{2}\right\rceil +1\right)n_{0}^r +\frac{1}{r!}\left(\left\lfloor \frac{k_t}{2}\right\rfloor -1\right) n^r +o(n^{r-1}). \label{eq4}
\end{align}

Combining \eqref{eq4} with \eqref{eq3} yields
$$\left(\binom{k_t}{r}-\binom{\left\lfloor \frac{k_t}{2}\right\rfloor -1}{r}\right) n_0+\frac{1}{r!}\left(\left\lceil \frac{k_t}{2}\right\rceil +1\right)n_{0}^r \geq \frac{1}{r!}n^r+o(n^{r-1}).$$
Since $\binom{k_t}{r}-\binom{\left\lfloor {k_t}/{2}\right\rfloor -1}{r}>0$, we have
$$\left(\binom{k_t}{r}-\binom{\left\lfloor \frac{k_t}{2}\right\rfloor -1}{r}+\frac{1}{r!}\left(\left\lceil \frac{k_t}{2}\right\rceil +1\right)\right)n_{0}^r \geq \frac{1}{r!}n^r+o(n^{r-1}).$$
This implies $n_0\geq c n+o(1)$, where $c$ is a constant depending only on $k_t$ and $r$. In particular, we may have $n_0\geq (k_1+k_2+\cdots+k_t)\binom{k_t}{\lfloor k_t/2\rfloor}$ as long as $n$ is large enough. On the other hand, $V(P)$ has $\binom{k_t}{\left\lfloor{k_t}/{2}\right\rfloor}$ subsets of $\lfloor k_t/2\rfloor$ vertices. So by the Pigeonhole principle, $G-P$ has at least $k_1+k_2+\cdots+k_t$ vertices that are adjacent to $\left\lfloor {k_t}/{2}\right\rfloor$ common vertices in $P$ and, hence, (\ref{U}) follows.

Finally, suppose to the contrary that $F'$ is a copy of $F-P_{k_t}$ in $G-U$. Since $|V(F')|=k_1+k_2+\cdots+k_{t-1}$, so by (\ref{U}), $|(\cap_{u\in U}N_{G}(u))\cap V(G-U-F')|\geq k_t$. This implies that we could construct a copy of $P_{k_t}$ in $G-F'$ by alternately using the vertices in $U$ and their common neighbors. Hence, $G$ contains a copy of $F$, which is a contradiction. Therefore, $G-U$ is $(F-P_{k_t})$-free, which completes our proof.\end{proof}

Let $U'$ be a set of $\left\lfloor {k_t}/{2}\right\rfloor$ vertices of degree $n-1$ in $V(G_F(n))$. Then $G_F(n)-U'\cong G_{F-P_{k_t}}(n-\left\lfloor {k_t}/{2}\right\rfloor)$ by the choice of $k_t$. So by the induction hypothesis, $\mathcal{N}(S_r,G-U)\leq \mathcal{N}(S_r,G_F(n)-U')$ for any fixed $r$ since $G-U$ is $(F-P_{k_t})$-free.

We also note that $\mathcal{N}(S_r,G)=ex(n,S_r,F)\geq \mathcal{N}(S_r,G_F(n))$. Recall that $\mathcal{N}_k(S_r, G-U, U)$ (resp. $\mathcal{N}_k(S_r, U, G-U)$) is the number of those $S_r$'s which have the center in $G-U$ (resp. $U$) and exactly $k$ leaves in $U$ (resp. $G-U$). Then we have
\begin{align*}
&~\mathcal{N}(S_r,G_F(n)-U')\\
\geq &~ \mathcal{N}(S_r,G-U)\\
= &~ \mathcal{N}(S_r,G)-\mathcal{N}(S_r,U)-\sum_{i=1}^{r} \mathcal{N}_i(S_r, G-U, U)-\sum_{i=1}^{r} \mathcal{N}_i(S_r, U, G-U)\\
\geq &~ \mathcal{N}(S_r,G_F(n))- \mathcal{N}(S_r, K_{\lfloor k_t/2\rfloor})-\sum_{i=1}^{r} \binom{\left\lfloor \frac{k_t}{2}\right\rfloor}{i} \mathcal{N}(S_{r-i}, G-U)\\
&~ -\sum_{i=1}^{r} \binom{n-\left\lfloor \frac{k_t}{2}\right\rfloor}{i} \mathcal{N}(S_{r-i},K_{\lfloor k_t/2\rfloor})\\
\geq &~ \mathcal{N}(S_r,G_F(n))- \mathcal{N}(S_r, K_{\lfloor k_t/2\rfloor})-\sum_{i=1}^{r} \binom{\left\lfloor \frac{k_t}{2}\right\rfloor}{i} \mathcal{N}(S_{r-i}, G_F(n)-U')\\
&~ -\sum_{i=1}^{r} \binom{n-\left\lfloor \frac{k_t}{2}\right\rfloor}{i} \mathcal{N}(S_{r-i},K_{\lfloor k_t/2\rfloor})\\
= &~\mathcal{N}(S_r, G_F(n)-U'),
\end{align*}
where the second `$\geq$' holds because $U\subseteq K_{\lfloor k_t/2\rfloor}$  and the last `$=$' holds by a direct calculation.
This implies $\mathcal{N}(S_r,G_F(n)-U')= \mathcal{N}(S_r,G-U)$. So, again by the induction hypothesis, $G-U\cong G_F(n)-U'$. Moreover, $G\subseteq G_F(n)$, which means $G\cong G_F(n)$, as desired.

2). We now prove the second part of the theorem.  Let $G$ be an $F$-free graph satisfying $\mathcal{N}(S_r,G)=ex(n,S_r,F)=ex(n,S_r,tP_3)$. To prove our result, it is sufficient to show $G\cong G_{tP_3}(n)$ when $n$ is large enough. We apply induction on $t$. If $t=2$, then the assertion  holds directly by Lemma \ref{lem2p}. Assume now $t\geq 3$.

In the proof of 1), taking the role of $k_t$ by 3, $P$ by the subgraph induced by an arbitrary $P_3$ and $U$ by a single vertex,  we can obtain the following lemma.
\begin{lem}\label{cla2}
If $n$ is large enough, then $V(P)$ contains a vertex $u$ such that
\begin{equation}\label{u}
\left|N_{G}(u)\cap V(G-P)\right|\geq 3t.
\end{equation}
Moreover, $G-u$ is $(t-1)P_3$-free.
\end{lem}
\begin{proof}
 Let $\mathcal{N}_k(S_r, G-P, P)$ (resp. $\mathcal{N}_k(S_r, P, G-P)$) be the number of those $S_r$'s which have the center in $G-P$ (resp. $P$) and exactly $k$ leaves in $P$ (resp. $G-P$). Then by the induction hypothesis and the similar argument as Lemma \ref{cla1}, we have \begin{align}
&~\mathcal{N}_r(S_r,G-P,P)+\mathcal{N}_r(S_r,P,G-P)\notag\\
\geq&~ \mathcal{N}(S_r,G_{tP_3}(n))-\mathcal{N}(S_r,G_{(t-1)P_3}(n-3))-3 \binom{2}{r}\notag\\
&~-\sum_{i=1}^{r-1}\binom{3}{i}\mathcal{N}(S_{r-i},G_{(t-1)P_3}(n-3))-\sum_{i=1}^{r-1} 3 \binom{n-3}{i} \binom{2}{r-i}\notag\\
=&~\frac{1}{r!}n^r+o(n^{r-1}).\label{eqp3}
\end{align}
Suppose to the contrary that each vertex in $P$ has at most $3t-1$ neighbors in $G-P$. Then $\mathcal{N}_r(S_r,G-P,P)\leq (3t-1)\binom{3}{r}$ and $\mathcal{N}_r(S_r,P,G-P)\leq 3\binom{3t-1}{r}$, which is a contradiction to \eqref{eqp3} when $n$ is large enough. Hence, $P$ contains a vertex $u$ having at least $3t$ neighbors in $G-P$.

Suppose to the contrary that $G-u$ contain a $(t-1)P_3$. By (\ref{u}) we have
$$\left|N_{G}(u)\cap V(G-P-(t-1)P_3)\right|\geq \left|N_{G}(u)\cap V(G-P)\right|-\left| V((t-1)P_3)\right|\geq 3.$$ This means that we could construct a copy of $P_{3}$ in $G-(t-1)P_3$ consisting of $u$ and two of its neighbors, which is a contradiction.
\end{proof}

Let $u'$ be a vertex with degree $n-1$ in $V(G_{tP_3}(n))$. Then $G_{tP_3}(n)-u'\cong G_{(t-1)P_3}(n-1)$. By the induction hypothesis,
\begin{align}\label{eqp4}
\mathcal{N}(S_{i},G-u)\leq \mathcal{N}(S_{i},G_{tP_3}(n)-u')
\end{align}
for $i\in\{r,r-1\}$ as $G-u$ is $(t-1)P_3$-free. Note that $\mathcal{N}(S_r,G)=\mathcal{N}(S_r,G-u)+\mathcal{N}_r(S_r,u,G-u)+\mathcal{N}_1(S_r,u,G-u)$, where $\mathcal{N}_r(S_r,u,G-u)$ is the number of the $S_r$'s with $u$ as the center and $\mathcal{N}_1(S_r,u,G-u)$  the number of the $S_r$'s with $u$ as a leaf. Then we have
\begin{align*}
~&\mathcal{N}(S_r,G_{tP_3}(n)-u')\\
\geq&~ \mathcal{N}(S_r,G-u)\\
= &~ \mathcal{N}(S_r,G)-\mathcal{N}_r(S_r,u,G-u)-\mathcal{N}_1(S_r,u,G-u)\\
\geq&~ \mathcal{N}(S_r,G)-\binom{d_G(u)}{r}-\mathcal{N}(S_{r-1},G-u)\\
\geq&~ \mathcal{N}(S_r,G_{tP_3}(n))- \binom{n-1}{r}-\mathcal{N}(S_{r-1},G_{tP_3}(n)-u')\\
= &~ \mathcal{N}(S_r, G_{tP_3}(n)-u'),
\end{align*}
where the last `$\geq$' holds by \eqref{eqp4}. Hence, $\mathcal{N}(S_r,G_{tP_3}(n)-u')=\mathcal{N}(S_r,G-u)$. So, again by the induction hypothesis, $G-u\cong G_{tP_3}(n)-u'$, which implies $G\subseteq G_{tP_3}(n)$. Then $G\cong G_{tP_3}(n)$ by $\mathcal{N}(S_r,G)=ex(n,S_r,tP_3)$, as desired.

Now we finish the proof of Theorem \ref{th1}. \qed

Finally, we remark that the extremal graphs for $ex(n,K_2,M_h)$ \cite{Erdos2}, $ex(n,K_2, F)$ ($F\neq tP_3$) \cite{Lidicky},  $ex(n, S_r, P_k)$ \cite{Gyori}, $ex(n,K_s,F)$ ($F$ consists of paths of even order) \cite{Zhu}, $ex(n,K_2, \mathcal{L}_{n,k})$
\cite{Wang2,Ning}, $ex(n,K_s, \mathcal{L}_{n,k})$, $ex(n,K^*_{s,t}, \mathcal{L}_{n,k})$ \cite{Zhang} are all isomorphic to $G_F(n)$. On the other hand, it is known that all the graphs involved in the above enumeration, i.e., $K_s, S_r$ and $K^*_{s,t}$, have a common property that a shifting operation does not decrease the numbers of their copies in an $F$-free graph. This leads to a natural problem, as follows:
\begin{pb}
Let $J$ be a graph such that the number of $J$'s in any graph $G$ does not decrease by shifting operation on $G$. Is the extremal graph for $ex(n,J,F)$ isomorphic to $G_F(n)$ for any linear forest $F$?  
\end{pb}

\section{Acknowledgements}
This work was supported by the National Natural Science Foundation of China [Grant numbers, 11971406, 12171402]

\end{document}